\documentclass[11pt,english]{article}

\usepackage[a4paper,
            left=0.7in,
            right=0.7in,
            top=1in,
            bottom=1.2in,
            footskip=.25in]{geometry}

\usepackage[utf8]{inputenc}
\usepackage[english]{babel}
\usepackage{pdfpages} 
\usepackage{graphicx}
\usepackage{siunitx}
\usepackage{hyperref}
\usepackage{url}
\usepackage{amsmath}
\usepackage{cleveref}
\usepackage{amssymb}
\usepackage{amsthm}
\usepackage{cases}
\usepackage{yfonts}
\usepackage{tikz}
\usepackage{anyfontsize}
\usepackage{appendix}
\usepackage{enumitem}
\usepackage{amsfonts}
\usepackage{dsfont}
\usepackage{comment}
\usepackage{mathtools}

\usepackage{bbm}

\usepackage{tikz}

\newcommand{\br}{\mathbb{R}}
\newcommand{\bz}{\mathbb{Z}}

\newcommand{\bt}{\mathbb{T}}

\newcommand{\be}{\beta}

\renewcommand{\(}{\left(}
\renewcommand{\)}{\right)}

\newcommand{\na}{\nabla}

\newcommand{\pa}{\partial}

\newtheorem{thm}{Theorem}[section]
\newtheorem{lem}[thm]{Lemma}
\newtheorem{prop}[thm]{Proposition}

\newtheorem{rem}[thm]{Remark}
\newtheorem{defn}[thm]{Definition}

\numberwithin{equation}{section}

\usepackage{scalerel,stackengine}
\stackMath
\newcommand\reallywidehat[1]{%
\savestack{\tmpbox}{\stretchto{%
  \scaleto{%
    \scalerel*[\widthof{\ensuremath{#1}}]{\kern-.6pt\bigwedge\kern-.6pt}%
    {\rule[-\textheight/2]{1ex}{\textheight}}
  }{\textheight}%
}{0.5ex}}%
\stackon[1pt]{#1}{\tmpbox}%
}
\usepackage{calrsfs}
\DeclareMathAlphabet{\pazocal}{OMS}{zplm}{m}{n}

\makeatletter
\@namedef{subjclassname@2020}{%
  \textup{2020} Mathematics Subject Classification}
\makeatother

\title{Backward problem for the 1D ionic Vlasov-Poisson equation}

\author{Antoine Gagnebin
\thanks{ETH Z\"urich. Email: \textsf{antoine.gagnebin@math.ethz.ch}}}
\date{\today}

\providecommand{\subjclass}[1]
{
  \small	
  \textbf{AMS subject classifications.} #1
}

\begin{document}

\maketitle

\begin{abstract} In this paper, we study the backward problem for the one-dimensional Vlasov-Poisson system with massless electrons, and we show the Landau damping by fixing the asymptotic behaviour of our solution. 
\end{abstract}


\section*{Introduction}

This work concerns the study of plasma physics.  Plasma is considered the fourth state of matter,  and it consists of an ionised gas in which electrons have separated from atoms due to high temperatures (see, e.g.  \cite{plasma_2003, Ryutov_1999, Villani_notes}).  Since ions are much larger and heavier than electrons, it is natural to consider different timescales when we want to track the evolution of the two species. 

A classical model for the evolution of electrons is the Vlasov-Poisson (VP) system. It is a well-known model in kinetic theory. It describes the motion of electrons in plasma when we neglect collisions and magnetic effects and consider the ions as a stationary background. This last assumption makes sense because the ions are much heavier than the electrons. Therefore, they move slowly compared to the motion of electrons and can be considered fixed. The one-dimensional (VP) system reads as follows: 
\begin{eqnarray}
\label{VP_intro}
\mbox{(VP) } := 
    \left\{
    \begin{array}{l}
    \partial_t f(t,x,v) + v \cdot \pa_x  f (t,x,v)  + E(t,x) \cdot \pa_v f(t,x,v) = 0, \\
    E (t,x) = - \pa_x U (t,x), \\
     \pa_{xx}^2 U (t,x) =  \rho_i - \rho (t,x),\\
    \rho(t,x) = \int_{\mathbb{R}} f(t,x,v) \ dv.
    \end{array}
    \right.
\end{eqnarray}
The unknown $f=f(t,x,v)$ represents the distribution function of electrons at time $t$, position $x$, and velocity $v$,  where $(t,x,v) \in \br \times \bt \times \br$.   The density of electrons $\rho = \rho (t,x)$ is the average in the velocity variable.  The density of ions, as said before, is fixed, i.e. $\rho_i = 1$.  The electric field $E = E(t,x)$ is generated by the collective behaviour of electrons and by the stationary background of ions.  The (VP) system has been extensively studied over the past decades, and a vast literature of results is already available.  For global well-posedness of classical solutions on the whole space,  we refer to the work of S. V. Iordanskii \cite{Iordanskii} in dimension one , S. Ukai and T. Okabe \cite{Ukai_Okabe} in dimension two, and independently by P.-L. Lions and B. Perthame \cite{Lions_Perthame}, and K. Pfaffelmoser \cite{Pfaffelmoser} in the three-dimensional case, (see also \cite{Bardos_Degond_2, Schaffer}).  For weak solutions on the whole space, we refer to the work of A. A. Arsenev \cite{Arsenev}, (see also \cite{Bardos_Degond_1, Bardos_Degond_Golse, Horst_Hunze}). J. Batt and G. Rein \cite{Batt_Rein} proved the global well-posedness on the torus,  (see also \cite{Chen, Pallard}).  Finally, Loeper \cite{Loeper} made a major improvement to the uniqueness theory,  (see also \cite{Miot}).

On the other hand, if we look at the evolution of ions in the plasma, we cannot consider the electrons as a fixed background.  Since  electrons move faster than  ions, it becomes relevant to consider electron-electron collisions in the model.  That is,  electrons are modelled by a collisional kinetic model and it is natural to assume that the distribution of electrons rapidly reaches a thermodynamic equilibrium. In summary, in the ion's timescale, the resulting density of electrons is given by a Maxwell-Boltzmann distribution.  Under these assumptions, we can derive the Vlasov-Poisson system for massless electrons (VPME), sometimes referred to as ionic Vlasov-Poisson or Vlasov-Poisson equation for ions with massless thermalized electrons. We refer to \cite{Bardos} for a rigorous derivation of (VPME) and 
\cite{Griffin_Iacobelli_summary} for more details on this model.  The (VPME) system consists of a Vlasov equation coupled with a nonlinear Poisson equation modelling how the electric potential is generated by the distribution of the ions and the Maxwell-Boltzmann distribution of the electrons.

In this paper, we are interested in the one-dimensional (VPME) on the torus $\bt$, which is given by:
\begin{eqnarray}
\label{VPME_intro}
\mbox{(VPME) } := 
    \left\{
    \begin{array}{l}
    \partial_t f(t,x,v) + v \cdot \pa_x  f(t,x,v)  + E (t,x) \cdot \pa_v f(t,x,v) = 0, \\
    E (t,x) = - \pa_x U (t,x), \\
     \pa_{xx}^2 U (t,x) = e^{U (t,x)} - \rho (t,x),\\
    \rho(t,x) = \int_{\mathbb{R}} f(t,x,v) \ dv.
    \end{array}
    \right.
\end{eqnarray}
The unknown $f=f(t,x,v)$ represents the distribution function of ions at time $t$, position $x$, and velocity $v$,  where $(t,x,v) \in \br \times \bt \times \br$.   The density of ions $\rho = \rho (t,x)$ is the average in the velocity variable.  The electric field $E = E(t,x)$ is generated by the collective behaviour of ions and by the thermal equilibrium of electrons. 

The (VPME) system has been less studied due to the additional nonlinearity in the Poisson equation.  The first work on global weak solutions is from F.  Bouchut \cite{Bouchut}. He constructed weak solutions globally in time in the whole space in dimension three. In the one-dimensional setting, weak solutions were constructed globally in time by D.  Han-Kwan and M. Iacobelli \cite{Han-Kwan_Iacobelli}. The global well-posedness has then been proved by M. Griffin-Pickering and M. Iacobelli in the case of the whole space in dimension three \cite{Griffin_Iacobelli_R}  and of the torus in dimension two and three \cite{Griffin_Iacobelli_torus}.  They showed existence of strong solutions for measure initial data with bounded moments and uniqueness with bounded density.  Recently, L. Cesbron and M. Iacobelli showed the global well-posedness of (VPME) in bounded domains \cite{Cesbron_Iacobelli}. 

This work aims to investigate  the backward problem of the system (\ref{VPME_intro}) by fixing the asymptotic distribution of $f$. Namely, let $f^*$ be a given distribution function that satisfies some decay and smoothness properties (see Definition \ref{def_f*} below), then we assume that for large $t$, our solution $f$ is close to the solution of the free transport $f^*(x-vt, v)$, i.e.
\begin{equation}
\label{asymptotic_cond_intro}
	\| f(t,x,v) - f^*(x-vt, v) \|_{L^{\infty} (\bt \times \br)} \longrightarrow 0 \qquad \mbox{as } t\rightarrow + \infty.
\end{equation}
It means that the flow induced by the electric field of the system (\ref{VPME_intro}) (see Definition \ref{def_flow}) is asymptomatically free.  We refer to the backward problem because we replace the usual initial datum $f^0(x,v) = f(0,x,v)$ by imposing an asymptotic state for our solution $f$. The idea behind assumption (\ref{asymptotic_cond_intro}) is the discovery of L. Landau in \cite{Landau}. He proved that the electric field of (VP) linearised around a Maxwellian homogeneous equilibrium is damped as time goes to infinity. Therefore, when the electric field of the system (\ref{VP_intro}) vanishes, it transforms into a transport equation.

The Landau damping, which refers to the decay of the electric field for large times, is a well-known collisionless relaxation phenomenon in kinetic theory. Several works extend and complete the pioneering result of L. Landau \cite{Landau}.  First, O. Penrose \cite{Penrose} extended the result of L. Landau for general spatially homogeneous equilibria.  Then, C. Mouhot and C. Villani proved, in their celebrated work \cite{CM_CV}, the Landau damping for the nonlinear (VP) system. They treated the perturbative regime\footnote{In this work, we refer to a perturbative regime as a solution of the form $\mu(v) + f(t,x,v)$,  where $\mu$ is a spatial homogeneous equilibrium and $f$ is a a small perturbation. } for the Cauchy problem with analytic and Gevrey initial data.
Their proof was then simplified and extended to Gevrey-$\frac{1}{3}$ initial data by J.  Bedrossian, N.  Masmoudi, and C.  Mouhot in \cite{BMM_13}. Recently, E. Grenier, T. Nguyen, and I. Rodianski gave a shorter proof for analytic and Gevrey initial condition in \cite{GNR}.  All these results mentioned above concern the forward problem, which provides the solution for the Cauchy problem for a given initial datum $f^0(x,v) = f(0,x,v)$, in a perturbative regime. 

Concerning the backward problem,  which provides the solution to the scattering problem for a given $f^*$, E.  Caglioti and C.  Maffei \cite{EC_CM} proved the damping of the electric field for the nonlinear (VP) in the scattering setting (\ref{asymptotic_cond_intro}).  Then, H. J. Hwang and J. -L. L. Vel\'{a}zquez \cite{Hw_Ve} gave a more sophisticated proof of the damping and obtained a larger class of possible asymptotic limits for $f^*$.  Recently,  D. Benedetto, E. Caglioti, and S. Rossi \cite{BCR} studied the backward and forward Landau damping problem for the Vlasov-HMF equation and compared these two approaches.

Very recently,  M. Iacobelli and the author  \cite{GI} proved the Landau damping for (VPME) on the torus for the Cauchy problem. They showed the decay of the electric field as time goes for analytic and Gevrey initial data in a perturbative regime.  Concerning the Landau damping for (VPME) on the whole space, recently, L. Huang, Q.-H. Nguyen, Y. Xu \cite{huang_sharp_2022} proved the decay of the density in dimension three and higher. The same group also treated the two-dimensional case in \cite{huang_2d}.

In this paper, we extend the method of E.  Caglioti and C.  Maffei \cite{EC_CM} to the case of (VPME).  The main ingredients in their work are the study of the flow from a Lagrangian description, the use of an iterative scheme on the Vlasov equation and a fixed point argument on the electric field.   We use the same techniques for our proof but with the added difficulty of dealing with a nonlinear coupling for the Poisson equation. To tackle this difficulty,  we use a decomposition of the electric field introduced by D.  Han-Kwan and M. Iacobelli in \cite{Han-Kwan_Iacobelli}  (see also \cite{Griffin_Iacobelli_R, Griffin_Iacobelli_torus}, where they used the same decomposition). The paper is thus organised as follows. In Section \ref{section_def} we give some definitions needed for the clarity of the paper, and we state our main theorem.  In Section \ref{section_E} we recall some general facts about the electric field of the system (\ref{VPME_intro}) and its decomposition.  We also show an important result on the potential $U$.  In Section \ref{section_proof} we prove the main result of this paper. Finally, in Section \ref{section_unstable} we give a class of stationary solutions, which are unstable in the weak topology.

\section{Definitions and main theorems}
\label{section_def}
We focus on solutions of the one-dimensional (VPME) on the torus $\bt$, that is 
\begin{eqnarray}
\label{VPME}
    (\mbox{VPME})=
    \left \{
    \begin{array}{l}
    \partial_t f + v \cdot \pa_x  f  + E \cdot \pa_v f = 0, \\
    E  = - \pa_x U, \\
     \pa_{xx}^2 U = e^{U} - \rho,\\
    \rho = \int_{\mathbb{R}} f \ dv.
    \end{array}
    \right.
\end{eqnarray}
We assume that $f$ satisfies the asymptotic condition,
\begin{equation}
\label{asymptotic_cond}
	\| f(t,x,v) - f^*(x-vt, v) \|_{L^{\infty} (\bt \times \br)} \longrightarrow 0 \qquad \mbox{as } t\rightarrow + \infty,
\end{equation}
for a given asymptotic datum $f^*$ as in Definition \ref{def_f*}.

We  work on the phase space $\Omega = \bt \times \br$. We identify the torus $\bt$  to the interval $[-\frac{1}{2}, \frac{1}{2})$ with periodic boundary conditions. If $z=(x,v)$ and $z'=(x',v') \in \Omega$, we define the norm
\begin{equation*}
	\vert z - z' \vert := \vert x-x' \vert_\bt + \vert v-v' \vert,
\end{equation*}
where $\vert x-x' \vert_\bt$ is understood as 
\begin{equation*}
	\vert x-x' \vert_\bt := \min_{k \in \bz} \vert x-x' + k \vert,
\end{equation*}
but we will omit the subscript $\bt$ for simplicity.

As said in the introduction, we use the decomposition of the electric field given in \cite{Han-Kwan_Iacobelli}.
\begin{defn}
\label{def_E}
	We define the two vector fields $\bar{E}$ and $\widetilde{E}$ such that $E=\bar{E} + \widetilde{E}$, 
\begin{equation*}
	\bar{E} = - \pa_x \bar{U},  \qquad \mbox{and} \qquad \widetilde{E} = - \pa_x \widetilde{U},
\end{equation*}
where $\bar{U}$ and $\widetilde{U}$ solve
\begin{equation*}
	 \pa_{xx}^2 \bar{U} = 1 - \rho,  \qquad \mbox{and} \qquad \pa_{xx}^2 \widetilde{U} = e^{\bar{U}+ \widetilde{U}} -1.
\end{equation*}
\end{defn}

In the next definition, we recall the concept of the characteristic curves and the flow induced by the electric field.
\begin{defn}
\label{def_flow}
We define the characteristic curves $X(t,x,v)$ and $V(t,x,v)$ of the system (\ref{VPME}) as follows
\begin{eqnarray}
\label{cc_VPME}
    \left \{
    \begin{array}{l}
    \dot{X} (t,x,v) = V(t,x,v),\\
    \dot{V}(t,x,v) = \bar{E}(t,X(t,x,v)) + \widetilde{E}(t,X(t,x,v)),\\
    \lim_{t \to \infty} X(t,x,v) - V(t,x,v)t = x,\\
    \lim_{t \to \infty} V(t,x,v) = v,
    \end{array}
    \right.
\end{eqnarray}
where $\bar{E}$ and $\widetilde{E}$ satisfies Definition \ref{def_E}. Moreover we define the flow of the characteristic curves (\ref{cc_VPME}) as $\phi_t (x,v) = (X(t,x,v), V(t,x,v))$. 
\end{defn}

\begin{rem}
	We observe that the characteristic curves are defined with an asymptotic condition instead of an initial condition.
\end{rem}

\begin{defn}
\label{homogeneous}
	We say that a solution $f$ of (\ref{VPME}) becomes homogeneous if there exists a function $h \in L_v^1 (\br)$ such that $f$ converges weakly to $h$.  That is for every test functions $\varphi \in C_c^0 (\Omega)$ we have 
	\begin{equation*}
		\lim_{t \rightarrow \infty} \int_\Omega \varphi (x,v) f(t,x,v) \ dx  \ dv = \int_\Omega \varphi (x,v) h(v) \ dx  \ dv.
	\end{equation*}
\end{defn}

In the next definition, we state the necessary assumptions for the asymptotic distribution $f^*$.
\begin{defn}
\label{def_f*}
Let $a,  a_1$,  $a_2$ and $\alpha$ be positive constants, we say that $f^* \in S_{a,a_1,a_2}$, if $f^* \geq 0$,
\begin{equation}
\label{Fourier transform f*}
	\vert \widehat{f^*}(k,\eta)  \vert \leq  \frac{e^{-12}}{4} \frac{1}{1 +\vert k \vert^\alpha} e^{-a\vert \eta \vert},
\end{equation}
where $\alpha \in (0,1)$ is a positive constant such that $\sum_{k=1}^\infty  \vert k \vert^{-(1 +\alpha)} \leq a_1$ and $\widehat{f^*}$ is the Fourier transform of $f^*$,  i.e.
\begin{equation*}
	 \widehat{f^*}(k,\eta) =  \int_\bt \int_\br f^*(x,v) e^{-2 \pi ikx} e^{-i \eta v} \ dx \ dv,
\end{equation*}
and 
\begin{equation}
\label{bdd f*}
	\vert f^*(x,v) \vert \leq \frac{a_2}{1+v^4}.
\end{equation}
\end{defn}

\begin{rem}
The property (\ref{Fourier transform f*}) requires that $f^*$ is  Hölder continous with exponent $\alpha$ in the $x$-variable and analytic in the $v$-variable. The decay property (\ref{bdd f*}) ensures that the density $\rho$ of our solution will remain in $L^1 \cap L^\infty (\bt)$. 
\end{rem}

To show the decay of the electric field,  we use the following norm which encodes the exponential decay in time if the norm is finite.
\begin{defn}
\label{def_norm}
Let $t_0 \geq 0$,  $a$ be the constant given in Definition \ref{def_f*}, and $F(t,x)$ be a vector field. We define the norm
	\begin{equation*}
		\|  F  \|_{a,t_0} := \sup_{t \geq t_0} e^{at} \|  F(t)  \|_{L^{\infty}(\bt)}.
	\end{equation*}
\end{defn}

With these definitions above, we can now state our main Theorem.

\begin{thm}
\label{thm1}
	Let $f^* \in S_{a,a_1,a_2}$ be a given asymptotic datum such that $a^2 \geq (112 a_2 + 3) (4 e^{12}+1)$, and let $t_0 \geq \max(0, \frac{1}{a} \log (8a_1 a_2))$. Then for $t \geq t_0$ there exists an electric field $E$ and a flow $\phi_t$ solution of (\ref{cc_VPME}) and there exists a weak solution of (\ref{VPME}) such that
\begin{itemize}
\item[i)] $f$ is given by 
\begin{equation}
	f(t,x,v) = f^*(\phi_t^{-1} (x,v)),
\end{equation}
\item[ii)] $f$ satisfies the  asymptotic condition (\ref{asymptotic_cond}),
\item[iii)] $f$ becomes homogeneous in the sense of Definition (\ref{homogeneous}) with weak limit
\begin{equation*}
	h(v) = \int_\bt f^* (x,v) \ dx.
\end{equation*}
\end{itemize}	
Moreover the electric field $E$ vanishes as t goes to infinity exponentially fast and we have the following bound
\begin{equation*}
	\|  E  \|_{a,t_0} \leq 16 a_1.
\end{equation*}
\end{thm}
The solution constructed above is not classical because $f^*$ is only Hölder continuous with respect to the $x$-variable. If we assume additional regularity on $f^*$ we have a classical solution.
\begin{thm}
\label{thm2}
	Let us suppose that $f^*$ satisfies the assumptions of  Theorem \ref{thm1}. Moreover assume that $f^* \in C^1(\Omega)$,
	\begin{equation}
\label{Fourier transform f* 2}
	\vert \widehat{f^*}(k,\eta)  \vert \leq \frac{e^{-12}}{4} \frac{a_1}{1 + k^2} e^{-a\vert \eta \vert},
\end{equation}
and
	\begin{equation}
		\vert \na_{x,v} f^* (x,v)  \vert \leq \frac{C}{1+v^2},
	\end{equation}
where $C$ is a positive constant. Then the solution $f= f^* \circ \phi_t^{-1}$ is a $C^1$ solution of (\ref{VPME}).
\end{thm}

\begin{rem}
	The proof of Theorem \ref{thm2} is similar to the proof of Theorem 3.4 in \cite{EC_CM} and will follow once we have proved Theorem \ref{thm1}. 
\end{rem}

\begin{rem}
We note that assumption (\ref{Fourier transform f* 2}) is the condition given in Definition 3.1 in \cite{EC_CM}.  For Theorem \ref{thm1}, it is not necessary to impose assumption (\ref{Fourier transform f* 2}) but it is sufficient to have our condition (\ref{Fourier transform f*}).
\end{rem}

\section{General facts on the electric field}
\label{section_E}
With the splitting of the electric field, $\bar{E}$ is the electric field in the classical (VP) system and can be written as a convolution between the fundamental solution of the Laplace equation on $\bt$ and the density $\rho$. That is, we have the following formula for $\bar{E}$,
\begin{equation*}
	\bar{E}(t,x) = - \int_\bt W'(x-y) \rho(t,y) \ dy,
\end{equation*}
where 
\begin{equation*}
	W(x) = \frac{x^2-\vert x \vert}{2}.
\end{equation*}
We observe that $\| W \|_{L^{\infty}(\bt)} \leq 1$ and $W$ is 1-Lipschitz. Moreover, if we identify\footnote{Without loss of generality we can identify $\bt$ to  $[-\frac{1}{2}, \frac{1}{2})$ or $[0,1)$ by doing a translation. We use $[0,1)$ in this case to simplify the notation of $W'$.} $\bt$ to $[0,1)$, we have
\begin{equation*}
	W'(x) = x - \frac{1}{2},  \qquad \mbox{and} \qquad W'(x+1) = W'(x). 
\end{equation*}
Note that $W'$ is discontinuous at $x=k$ for $k \in \mathbb{Z}$.  In addition,
\begin{equation}
\label{bound_W'}
	\| W' \|_{L^{\infty}(\bt)} \leq 1,
\end{equation}
and if the interval $[x,x']$ does not contain $k$ for any $k \in \bz$, then
\begin{equation}
\label{lipschitz_W'}
	\vert  W'(x) - W'(x')  \vert \leq \vert x-x' \vert.
\end{equation}
For $\widetilde{E}$ we do not have an explicit formula, but we can show that $\widetilde{E}$ is more regular that $\bar{E}$, see \cite[Section 3]{Griffin_Iacobelli_torus}, for references. In the following, we recall some results about $\bar{U}$ and $\widetilde{U}$. First, $\bar{U}$ is given by
\begin{equation*}
	 \bar{U} (t,x) = \int_\bt W(x-y) \rho(t,y) \ dy.
\end{equation*}
For $\rho \in L^1 \cap L^\infty (\bt)$,  we see that $\| \bar{U} \|_{L^{\infty}(\bt)} \leq 1$ and $ \bar{U}$ is Lipschitz.  Concerning $\widetilde{U}$,  again, we do not have an explicit formula, but we still have existence and unicity of a strong solution $\widetilde{U}$ with  bounds on the $L^\infty$ norm of $\widetilde{U}$ and its derivatives. This result is given by Lemma 2.2 in \cite{Han-Kwan_Iacobelli}.
\begin{lem}
\label{lemma widetilde U}
There exists a unique solution on $\bt$ of 
 \begin{equation*}
 	 \pa_{xx}^2 \widetilde{U} = e^{\bar{U}+\widetilde{U}} - 1,
 \end{equation*}
 and this solution satisfies the three following bounds,
 \begin{equation*}
 	\| \widetilde{U} \|_{L^{\infty}(\bt)} \leq 3 , \qquad \| \pa_x \widetilde{U} \|_{L^{\infty}(\bt)} \leq 2,  \qquad\| \pa_{xx}^2 \widetilde{U} \|_{L^{\infty}(\bt)} \leq 3.
 \end{equation*} 
\end{lem}
By comparison with Proposition 3.1 \cite{Griffin_Iacobelli_torus}, we observe that in dimension one, we have better regularity result on $\widetilde{U}$ than in higher dimension. 

The last goal of this section is to show a $L^\infty$ stability estimate between $ \partial_x \widetilde{U}$ and $\bar{U}$. Namely, we want to show the following.
\begin{prop}
\label{prop_L_infty_bound}
 Let $\widetilde{U}_1$ and $\widetilde{U}_2$ be two strong solutions of 
  \begin{equation}
    \label{equ_U''}
 	 \pa_{xx}^2 \widetilde{U}_i (x) = e^{\bar{U}_i (x) +\widetilde{U}_i (x)} - 1, \qquad \mbox{for } i=1,2.
 \end{equation}
Then
 \begin{equation}
 \label{U'_1 - U'_2 hat}
 	\|  \pa_x \widetilde{U}_1  - \pa_x \widetilde{U}_2  \|_{L^{\infty} (\bt)} \leq 4 e^{12} \| \bar{U}_1  - \bar{U}_2  \|_{L^{\infty} (\bt)}.
 \end{equation}
\end{prop}
 \begin{proof}
The beginning of the proof is analogue to the $L^2$ estimate given by Lemma 3.9  \cite{Griffin_Iacobelli_torus} in dimension two and three.  We redo it in dimension one for completeness and then we extend the $L^2$ estimate to a $L^\infty$ estimate. We have by (\ref{equ_U''})
\begin{equation*}
	\pa_{xx}^2 \widetilde{U}_1 - \pa_{xx}^2 \widetilde{U}_2 =   e^{\bar{U}_1+\widetilde{U}_1} - e^{\bar{U}_2 +\widetilde{U}_2} = e^{\bar{U}_1} \( e^{\widetilde{U}_1} - e^{\widetilde{U}_2}  \) + e^{\widetilde{U}_2} \( e^{\bar{U}_1} - e^{\bar{U}_2} \).
\end{equation*}
We multiply both sides of the last equality by $(\widetilde{U}_1 - \widetilde{U}_2) $ and we integrate on $\bt$, 
\begin{align*}
	\int_\bt (\pa_{xx}^2 \widetilde{U}_1 - \pa_{xx}^2 \widetilde{U}_2) (\widetilde{U}_1 - \widetilde{U}_2)  \ dx & = \int_\bt   e^{\bar{U}_1} \( e^{\widetilde{U}_1} - e^{\widetilde{U}_2}  \) (\widetilde{U}_1 - \widetilde{U}_2)  \ dx + \int_\bt e^{\widetilde{U}_2} \( e^{\bar{U}_1} - e^{\bar{U}_2} \) (\widetilde{U}_1 - \widetilde{U}_2)  \ dx.
\end{align*}
Next, by integration by parts, we have
\begin{multline}
\label{U_1 - U_2}
	- \int_\bt ( \pa_x \widetilde{U}_1 - \pa_x \widetilde{U}_2)^2 \ dx
	= \int_\bt   e^{\bar{U}_1} \( e^{\widetilde{U}_1} - e^{\widetilde{U}_2}  \) (\widetilde{U}_1 - \widetilde{U}_2)  \ dx + \int_\bt e^{\widetilde{U}_2} \( e^{\bar{U}_1} - e^{\bar{U}_2} \) (\widetilde{U}_1 - \widetilde{U}_2)  \ dx  =: I_1 + I_2. 
\end{multline}
We observe that ($e^x-e^y)(x-y)$ is always non-negative. Furthermore by the Mean Value Theorem applied to the function $x \mapsto e^x$, we have a lower bound
\begin{equation*}
	(e^x-e^y)(x-y) \geq e^{\min \{x,y \}} (x-y)^2.
\end{equation*}
Moreover, we have $e^x \geq e^{-\vert x \vert}$ for every $x$. Therefore, we can bound $I_1$ from below as follow,
\begin{align}
\label{bound_I1}
	I_1 & \geq \int_\bt   e^{- \vert \bar{U}_1 \vert} e^{\min \{ \widetilde{U}_1(x) ,\widetilde{U}_2(x) \} }   (\widetilde{U}_1 - \widetilde{U}_2)^2  \ dx \geq \int_\bt   e^{- \| \bar{U}_1 \|_{L^{\infty}}}  e^{- \vert \min \{ \widetilde{U}_1(x) ,\widetilde{U}_2(x) \} \vert}     (\widetilde{U}_1 - \widetilde{U}_2)^2  \ dx \nonumber \\
	 & \geq \int_\bt   e^{- \| \bar{U}_1 \|_{L^{\infty}}}  e^{- \max_i \| \widetilde{U}_i \|_{L^{\infty}} }     (\widetilde{U}_1 - \widetilde{U}_2)^2  \ dx \geq e^{-4} \int_\bt (\widetilde{U}_1 - \widetilde{U}_2)^2   \ dx,
\end{align}
where we used $\| \bar{U} \|_{L^{\infty}(\bt)} \leq 1 $ and $ \| \widetilde{U} \|_{L^{\infty}(\bt)} \leq 3$. For $I_2$ we have again by the Mean Value Theorem,
\begin{equation}
\label{ineq_exp}
	\left| e^x-e^y \right| \leq e^{\max \{ x,y \}} \vert x-y \vert.
\end{equation}
Therefore,
\begin{equation}
\label{bound_I2}
	I_2 \leq e^{ \| \widetilde{U}_2 \|_{L^{\infty}} + \max_i \| \bar{U}_i \|_{L^{\infty}} } \int_\bt \vert \bar{U}_1 - \bar{U}_2 \vert \vert \widetilde{U}_1 - \widetilde{U}_2 \vert  \ dx \leq e^4 \int_\bt \vert \bar{U}_1 - \bar{U}_2 \vert \vert \widetilde{U}_1 - \widetilde{U}_2 \vert  \ dx.
\end{equation}
Therefore, using (\ref{bound_I1}) and (\ref{bound_I2}) into (\ref{U_1 - U_2}), we get
\begin{align*}
	 \int_\bt ( \pa_x \widetilde{U}_1 - \pa_x \widetilde{U}_2)^2   \ dx & \leq - e^{-4} \int_\bt (\widetilde{U}_1 - \widetilde{U}_2)^2   \ dx + e^4 \int_\bt \vert \bar{U}_1 - \bar{U}_2 \vert \vert \widetilde{U}_1 - \widetilde{U}_2 \vert  \ dx \nonumber \\
	 & \leq - e^{-4} \int_\bt (\widetilde{U}_1 - \widetilde{U}_2)^2   \ dx + e^4 \int_\bt \be \vert \bar{U}_1 - \bar{U}_2 \vert^2  +  \frac{1}{\be} \vert \widetilde{U}_1 - \widetilde{U}_2 \vert ^2  \ dx,
\end{align*}
where we have used $\vert ab \vert \leq \be a^2 + \frac{1}{\be} b^2$ for any $\be >0$.  Then if we choose $\beta = 2e^8$ we obtain
\begin{align*}
	\frac{e^{-4}}{2} \int_\bt (\widetilde{U}_1 - \widetilde{U}_2)^2   \ dx + \int_\bt ( \pa_x \widetilde{U}_1 - \pa_x \widetilde{U}_2)^2   \ dx & \leq  2 e^{12} \int_\bt  \vert \bar{U}_1 - \bar{U}_2 \vert^2    \ dx,
\end{align*}
because $e^{-4} - \frac{e^4}{\beta} = \frac{e^{-4}}{2}$. With this last inequality, we have obtained the two following estimates
\begin{equation}
\label{L2_U}
	\|   \widetilde{U}_1  -  \widetilde{U}_2  \|_{L^{2} (\bt)} \leq 2 e^{8} \|   \bar{U}_1  -  \bar{U}_2  \|_{L^{2} (\bt)},
\end{equation}
and
\begin{equation}
\label{L2_U'}
	\|   \pa_x \widetilde{U}_1  - \pa_x \widetilde{U}_2  \|_{L^{2} (\bt)} \leq \sqrt{2} e^6 \|   \bar{U}_1  -  \bar{U}_2  \|_{L^{2} (\bt)}.
\end{equation}
Moreover, if we consider (\ref{equ_U''}), we have by (\ref{ineq_exp})
\begin{align*}
	\| \pa_{xx}^2 \widetilde{U}_1 - \pa_{xx}^2 \widetilde{U}_2  \|_{L^{2} (\bt)} & =   \Vert e^{\bar{U}_1+\widetilde{U}_1} - e^{\bar{U}_2 +\widetilde{U}_2}  \Vert_{L^{2} (\bt)}  \leq e^4 \Vert ( \bar{U}_1+\widetilde{U}_1) - (\bar{U}_2  + \widetilde{U}_2)  \Vert_{L^{2} (\bt)} ,
\end{align*}
where we used $\| \bar{U} \|_{L^{\infty}(\bt)} + \| \widetilde{U} \|_{L^{\infty}(\bt)} \leq 4$. Therefore, using (\ref{L2_U}), we find
\begin{align}
\label{L2_U''}
	\| \pa_{xx}^2 \widetilde{U}_1 - \pa_{xx}^2 \widetilde{U}_2  \|_{L^{2} (\bt)} & \leq e^4 \Vert  \bar{U}_1 - \bar{U}_2   \Vert_{L^{2} (\bt)}  + e^4 \Vert  \widetilde{U}_1 - \widetilde{U}_2   \Vert_{L^{2} (\bt)} \nonumber \\
	& \leq e^4 \Vert  \bar{U}_1 - \bar{U}_2   \Vert_{L^{2} (\bt)} + 2 e^{12}  \Vert  \bar{U}_1 - \bar{U}_2   \Vert_{L^{2} (\bt)} \nonumber \\
	& \leq 3 e^{12}  \Vert  \bar{U}_1 - \bar{U}_2   \Vert_{L^{2} (\bt)}.
\end{align}
Finally, with the $L^\infty$ Sobolev embedding, we get
\begin{align*}
	\|  \pa_x \widetilde{U}_1  - \pa_x \widetilde{U}_2  \|_{L^{\infty} (\bt)} & \leq \|  \pa_x \widetilde{U}_1  - \pa_x \widetilde{U}_2  \|_{W^{1,2} (\bt)} \leq \|   \pa_x \widetilde{U}_1  - \pa_x \widetilde{U}_2  \|_{L^{2} (\bt)} + \| \pa_{xx}^2 \widetilde{U}_1 - \pa_{xx}^2 \widetilde{U}_2  \|_{L^{2} (\bt)} \\
	& \leq \sqrt{2}e^6 \|   \bar{U}_1  -  \bar{U}_2  \|_{L^{2} (\bt)} + 3 e^{12}  \Vert  \bar{U}_1 - \bar{U}_2   \Vert_{L^{2} (\bt)} \leq 4 e^{12}  \Vert  \bar{U}_1 - \bar{U}_2   \Vert_{L^{2} (\bt)}\\
	& \leq 4 e^{12}  \Vert  \bar{U}_1 - \bar{U}_2   \Vert_{L^{\infty} (\bt)},
\end{align*}
where we used (\ref{L2_U'}),  (\ref{L2_U''}),and $L^\infty \subseteq L^2$ for bounded domains.
\end{proof}

\section{Proof of Theorem \ref{thm1} }
\label{section_proof}

The idea to show that the Landau damping occurs for the system (\ref{VPME}) with the condition (\ref{asymptotic_cond}) is an iterative scheme. More precisely, let $f^*$ be the asymptotic datum of Theorem \ref{thm1} and consider the sequence of linear problems
\begin{equation}
\label{VPME_lin_seq}
	 \partial_t f_{n+1} + v \cdot \pa_x  f_{n+1}  +( \bar{E}_n + \widetilde{E}_n ) \cdot \pa_v f_{n+1} = 0,
\end{equation}
with the asymptotic condition 
\begin{equation}
\label{asymptotic_cond_seq}
	\| f_{n+1}(t,x,v) - f^*(x-vt, v) \|_{L^{\infty} (\bt \times \br)} \longrightarrow 0 \qquad \mbox{as } t\rightarrow + \infty.
\end{equation}
For $n=0$ we define $ \bar{E}_0 = \widetilde{E}_0 = 0$ and for $n \geq 1$, 
\begin{eqnarray}
\label{sequence_electric_field}
 \left\{
    \begin{array}{l}
	\bar{E}_n = -\pa_x \bar{U}_n, \qquad \pa_{xx}^2 \bar{U}_n =1 - \rho_n,  \\
	\widetilde{E}_n = -\pa_x \widetilde{U}_n, \qquad \pa_{xx}^2 \widetilde{U}_n = e^{\bar{U}_n+ \widetilde{U}_n} -1, \\
	\rho_n = \int_\br f_n \ dv .
	\end{array}
	\right.
\end{eqnarray}
The iterative scheme reads as follows
\begin{eqnarray*}
    \left \{
    \begin{array}{l}
   0= \bar{E}_0\\
   0=  \widetilde{E}_0
    \end{array}
    \right \} \xrightarrow[(\ref{VPME_lin_seq})]{\mbox{solve}}  f_1 \longrightarrow \rho_1 \xrightarrow[(\ref{sequence_electric_field})]{\mbox{solve}} 
     \left \{
     \begin{array}{l}
   \bar{E}_1\\
    \widetilde{E}_1
    \end{array}
      \right \} \xrightarrow[(\ref{VPME_lin_seq})]{\mbox{solve}} f_2 \longrightarrow \rho_2 \xrightarrow[(\ref{sequence_electric_field})]{\mbox{solve}}
     \left \{
     \begin{array}{l}
   \bar{E}_2\\
    \widetilde{E}_2
    \end{array}
      \right \} \longrightarrow ...
\end{eqnarray*}
We show that the problem (\ref{VPME_lin_seq}) - (\ref{sequence_electric_field}) can be solved for all $n$ and that the pair of solutions $(f_n,  E_n)$ converges as $n$ goes to infinity to a solution of the system (\ref{VPME}) with the asymptotic condition (\ref{asymptotic_cond}).  The key element is to show that the map $E_n \mapsto E_{n+1}$ is a contraction with the norm given in Definition \ref{def_norm}.
\begin{prop}
	In the setting defined above and with the assumptions of Theorem \ref{thm1} we have 
	\begin{align}
	\label{bdd E_n}
		\|  E_n  \|_{a,t_0} & \leq 16 a_1 , \\
		\label{E_n Lipschitz}
		\vert E_n (t,x) - E_n (t, x') \vert & \leq (24 a_2+3) \vert x-x' \vert, \\
		\label{bdd E_n - E_n-1}
		\|  E_{n+1} - E_n  \|_{a,t_0} & \leq \frac{1}{2^n} 16 a_1 ,
	\end{align}
where $E_n = \bar{E}_n + \widetilde{E}_n$.
\end{prop}

In the next two subsections we prove this proposition by induction.

\subsection{Base case \texorpdfstring{$n=1$}{Lg}}
For $n=0$,  $ \bar{E}_0 = \widetilde{E}_0 = 0$ and $f_1$ solves the free transport equation, i.e. 
\begin{equation}
	 \partial_t f_{1} + v \cdot \pa_x  f_{1}  = 0.
\end{equation}
Therefore, $f_1 (t,x,v) = f^*(x-vt, v)$ and $\rho_1(t,x) = \int_\br f^*(x-vt, v) \ dv$. Next, we show $\|  E_1  \|_{a,t_0}  \leq 8 a_1$.  First, we bound $ \bar{E}_1$.  Using  equations (\ref{sequence_electric_field}), we have $ \pa_{x} \bar{E}_1 = \rho_1 - 1$. We can transform this on the Fourier side and we get the following relation
\begin{eqnarray*}
 \left\{
    \begin{array}{l l}
	k \widehat{\bar{E}_1}(t,k) = \widehat{\rho_1}(t,k) = \widehat{f^*}(k,kt) & \quad \mbox{for } k \ne 0, \\
	\widehat{\bar{E}_1}(t,k) = 0 & \quad \mbox{for } k = 0. 
	\end{array}
	\right.
\end{eqnarray*}
Therefore,
\begin{align*}
	\|  \bar{E}_1  \|_{L^{\infty} (\bt)} & \leq  \sum_{k \in \mathbb{Z}} \vert \widehat{\bar{E}_1}(k) \vert \leq  \sum_{k \in \mathbb{Z} \setminus \{ 0 \}} \frac{1}{\vert k \vert} \vert  \widehat{f^*}(k,kt) \vert  \leq   \sum_{k \in \mathbb{Z} \setminus \{ 0 \}} \frac{e^{-12}}{4} \frac{1}{\vert k \vert} \frac{1}{1 +\vert k \vert^\alpha} e^{-a\vert kt \vert} \\
	& \leq e^{-at}  \sum_{k \in \mathbb{Z} \setminus \{ 0 \}}  \frac{1}{\vert k \vert^{1+\alpha}}   \leq 4 a_1 e^{-at},
\end{align*}
where we used (\ref{Fourier transform f*}) and $\sum_{k=1}^\infty  \frac{1}{\vert k \vert^{1+\alpha}} \leq a_1$. Therefore, we get $\|  \bar{E}_1  \|_{a,t_0} \leq 4a_1$. 
For $\|  \widetilde{E}_1  \|_{a,t_0}$ we use the bound (\ref{U'_1 - U'_2 hat}) obtained in Proposition \ref{prop_L_infty_bound} to get
\begin{align*}
	\|  \widetilde{E}_1  \|_{L^{\infty} (\bt)} =\|  \pa_x \widetilde{U}_1  \|_{L^{\infty} (\bt)} \leq 4e^{12} \| \bar{U}_1  \|_{L^{\infty} (\bt)}.
\end{align*}
Then, using equation $\pa_{xx}^2 \bar{U}_1 =1 - \rho_1$, we have a relation between the Fourier transform of $\bar{U}_1$ and $f^*$ as before
\begin{eqnarray*}
 \left\{
    \begin{array}{l l}
	k^2 \widehat{\bar{U}_1}(t,k) = \widehat{\rho_1}(t,k) = \widehat{f^*}(k,kt) & \quad \mbox{for } k \ne 0, \\
	\widehat{\bar{U}_1}(t,k) = 0 & \quad \mbox{for } k = 0. \nonumber
	\end{array}
	\right.
\end{eqnarray*}
Therefore,
\begin{align*}
	\|  \widetilde{E}_1  \|_{L^{\infty} (\bt)} & \leq 4e^{12} \| \bar{U}_1  \|_{L^{\infty} (\bt)} \leq 4 e^{12} \sum_{k \in \mathbb{Z}} \vert \widehat{\bar{U}_1}(k) \vert \leq 4 e^{12} \sum_{k \in \mathbb{Z} \setminus \{ 0 \}} \frac{1}{\vert k \vert^2} \vert  \widehat{f^*}(k,kt) \vert \\
	& \leq 4 e^{12}  \sum_{k \in \mathbb{Z} \setminus \{ 0 \}} \frac{e^{-12}}{4} \frac{1}{\vert k \vert^2} \frac{1}{1 + \vert k \vert^\alpha} e^{-a\vert kt \vert} \leq e^{-at}  \sum_{k \in \mathbb{Z} \setminus \{ 0 \}}  \frac{1}{\vert k \vert^{1+\alpha}}   \leq 4 a_1 e^{-at},
\end{align*}
where we used again (\ref{Fourier transform f*}) and  $\sum_{k=1}^\infty  \frac{1}{\vert k \vert^{1+\alpha}} \leq a_1$.
Finally, we get $\|  \widetilde{E}_1  \|_{a,t_0} \leq 4a_1$. This shows that
\begin{equation}
	\label{bdd E_1}
	\|  E_1  \|_{a,t_0} \leq  \|  \bar{E}_1  \|_{a,t_0} + \|  \widetilde{E}_1  \|_{a,t_0}  \leq 8 a_1.
\end{equation}
We still need to prove that $E_1$ is Lipschitz. First, we show that $\rho_1 \in L^1 (\bt) \cap L^{\infty} (\bt)$. Indeed, with (\ref{bdd f*})
\begin{align*}
	\| \rho_1 (t)  \|_{L^{\infty} (\bt)}  = \sup_{x \in \bt} \int_\br f^*(x-vt, v) \ dv \leq  \sup_{x \in \bt} \int_\br \frac{a_2}{1+v^4} \ dv \leq 4a_2,
\end{align*}
and similarly,
\begin{align*}
	\| \rho_1 (t)  \|_{L^{1} (\bt)}  \leq 4a_2.
\end{align*}
Let $[x,x']$ be an interval that does not contain $k$ for any $k \in \bz$. Then for $y \in \bt $ and $y \notin [x,x']$ we have $[ (x-y), (x'-y) ]$ does not contain any $k \in \bz$. Therefore,  
\begin{multline*}
	\vert \bar{E}_1(t,x) - \bar{E}_1 (t,x') \vert \leq  \int_\bt  \vert W'(x-y) - W'(x'-y) \vert \rho_1(t,y) \ dy \\
	 \leq \int_{\{ y \in [x,x'] \}} \vert W'(x-y) - W'(x'-y) \vert \rho_1 (t,y) \ dy \\
	 + \int_{\{ y \notin [x,x'] \}} \vert W'(x-y) - W'(x'-y) \vert \rho_1 (t,y) \ dy.
	\end{multline*}
For the first integral we use (\ref{bound_W'}) and for the second integral we use (\ref{lipschitz_W'}). Thus,	
\begin{align}
\label{lipschitz E_1}
	\vert \bar{E}_1(t,x) - \bar{E}_1 (t,x') \vert &  \leq  \| \rho_1 (t) \|_{L^{\infty} (\bt)}  \int_{\{ y \in [x,x'] \}} 2 \vert W'(y) \vert \ dy + \vert x-x' \vert \int_{\{ y \notin [x,x'] \}} \rho_1 (t,y) \ dy  \nonumber \\
	& \leq 2 \| \rho_1 (t) \|_{L^{\infty} (\bt)}	 \int_x^{x'} dy + \vert x-x' \vert  \int_\bt \rho_1 (t,y) \ dy \nonumber \\
	& \leq 2  \| \rho_1 (t) \|_{L^{\infty} (\bt)} \vert x-x' \vert  + \vert x-x' \vert  \| \rho_1 (t) \|_{L^1 (\bt)} \leq 12a_2 \vert x-x' \vert.
\end{align}
For $\widetilde{E}_1$, the bound $\| \pa_x \widetilde{E}_1 \|_{L^{\infty} (\bt)} \leq 3$, given in Lemma \ref{lemma widetilde U} implies that $\widetilde{E}_1$ is 3-Lipschitz. Hence,
\begin{equation}
\label{Lip_E1}
	\vert E_1 (t,x) - E_1 (t,x') \vert  \leq (12 a_2 + 3) \vert x-x' \vert.
\end{equation}
In this section, we have shown that (\ref{bdd E_n}), (\ref{E_n Lipschitz}) and (\ref{bdd E_n - E_n-1}) hold for $n=0$ and $n=1$.

\subsection{Induction step \texorpdfstring{$n \Rightarrow n+1$}{Lg}}

Here, we prove that $E_{n+1}$ satisfies  (\ref{bdd E_n}), (\ref{E_n Lipschitz}) and (\ref{bdd E_n - E_n-1}) under the assumption that $E_n$ does.

Given $E_n$, we can solve (\ref{VPME_lin_seq}) using the characteristic curves:
\begin{equation*}
	  \left \{
    \begin{array}{l}
    \dot{X}_{n+1} (t) = V_{n+1}(t),\\
    \dot{V}_{n+1}(t) = \bar{E}_n(t) + \widetilde{E}_n(t) = E_n(t),\\
    \lim_{t \to \infty} X_{n+1}(t) - V_{n+1}(t)t = x,\\
    \lim_{t \to \infty} V_{n+1}(t) = v.
    \end{array}
    \right.
\end{equation*}
Since $E_n$ is Lipschitz and bounded by assumption,  we have the existence and uniqueness of a well defined flow $\phi_t (x,v) = (X_{n+1}(t,x,v), V_{n+1}(t,x,v))$.  By step 2 in the proof of \cite[Lemma 3.1]{EC_CM}, we have that the flow and its inverse are Hölder continuous.
We denote the inverse of the flow $\phi_t^{-1}$.  Since it is Hölder continuous we can define $f_{n+1}$ as 
\begin{equation*}
	f_{n+1}(t,x,v) = f^* (\phi_t^{-1} (x,v)).
\end{equation*}
Moreover, the asymptotic condition (\ref{asymptotic_cond_seq}) is satisfied. Indeed, we have using the bijectivity of $\phi$ and the asymptotic limit of $X_{n+1}$ and $V_{n+1}$, 
\begin{align}
\label{computation asymptotic}
	\left\Vert f_{n+1}(t,x,v) - f^*(x-vt,v) \right\Vert_{L^{\infty} (\bt \times \br)} & = \left\Vert f_{n+1}(t,x + vt,v) - f^*(x,v) \right\Vert_{L^{\infty} (\bt \times \br)} \nonumber \\
	& = \left\Vert f^* (\phi_t^{-1} (x + vt,v)) - f^*(x,v) \right\Vert_{L^{\infty} (\bt \times \br)} \nonumber \\
	& =  \left\Vert f^* (x + vt,v) - f^*(\phi_t (x,v)) \right\Vert_{L^{\infty} (\bt \times \br)} \nonumber \\
	& =  \left\Vert f^* (x + vt,v) - f^*(X_{n+1}(t,x,v), V_{n+1}(t,x,v)) \right\Vert_{L^{\infty} (\bt \times \br)} \xrightarrow[t \rightarrow \infty]{} 0.
\end{align}
The next lemma shows that $\rho_{n+1} \in L^{1} (\bt) \cap L^{\infty} (\bt)$.
\begin{lem}
\label{lemma bound_rho_n+1}
We have the two following bounds on the density $\rho_{n+1}$. 
\begin{equation}
\label{bound_rho_n+1}
	\| \rho_{n+1} \|_{L^{1} (\bt)} \leq 4 a_2 \qquad \mbox{and } \qquad \| \rho_{n+1} \|_{L^{\infty} (\bt)} \leq 10 a_2.
\end{equation}
\end{lem}
\begin{proof}
Since $(V_{n+1}, E_n)$ is divergence free, we have that the flow $\phi_t$ is measure preserving. Therefore, by the change of variable formula, we obtain
\begin{align*}
	\| \rho_{n+1} \|_{L^{1} (\bt)} = \int_\bt \int_\br f_{n+1}(t,x,v) \ dv  dx = \int_\bt \int_\br f^* (\phi_t^{-1} (x,v)) \ dv  dx =  \int_\bt \int_\br f^* (x,v) \ dv  dx,
\end{align*}
thus, with (\ref{bdd f*}), we get
\begin{equation*}
	\| \rho_{n+1} \|_{L^{1} (\bt)} \leq \int_\bt \int_\br \frac{a_2}{1+v^4} \ dv  dx \leq 4 a_2.
\end{equation*}
Similarly to step 4 in \cite[Lemma 3.1]{EC_CM}, we obtain $\| \rho_{n+1} \|_{L^{\infty} (\bt)} \leq 10 a_2$.  
\end{proof}
In the next lemma we show the Lipschitz property of the electric field $E_{n+1}$, i.e. inequality (\ref{E_n Lipschitz}).
\begin{lem}
\label{lemma Lipschitz E_n+1}
Let $E_{n+1} = \bar{E}_{n+1} + \widetilde{E}_{n+1}$. Then we have the following result
\begin{equation}
\label{Lipschitz E_n+1}
	\vert E_{n+1} (t,x) - E_{n+1} (t,x') \vert  \leq (24 a_2+3) \vert x-x' \vert.
\end{equation}
\end{lem}
\begin{proof}
First, we can write $\bar{E}_{n+1}$ as a convolution with $W'$ and $\rho_{n+1}$.   Then, we proceed in the same way as for $\bar{E}_1$ and we obtain as in (\ref{lipschitz E_1})
\begin{align*}
	\vert \bar{E}_{n+1}(t,x) - \bar{E}_{n+1} (t,x') \vert & \leq 2  \| \rho_{n+1} (t) \|_{L^{\infty} (\bt)} \vert x-x' \vert  + \vert x-x' \vert  \| \rho_{n+1} (t) \|_{L^1 (\bt)} \leq 24a_2 \vert x-x' \vert,
\end{align*}
where we used the two bounds given by (\ref{bound_rho_n+1}). For $\widetilde{E}_{n+1}$, we use  the bound $\| \pa_x \widetilde{E}_{n+1} \|_{L^{\infty} (\bt)} = \| \pa_{xx}^2 \widetilde{U}_{n+1} \|_{L^{\infty} (\bt)} \leq 3$, given in Lemma \ref{lemma widetilde U} to have 
\begin{equation*}
	\vert \widetilde{E}_{n+1} (t,x) - \widetilde{E}_{n+1} (t,x') \vert  \leq 3 \vert x-x' \vert.
\end{equation*}
Hence, we proved that $E_{n+1}$ is Lipschitz,
\begin{equation*}
	\vert E_{n+1} (t,x) - E_{n+1} (t,x') \vert  \leq (24 a_2+3) \vert x-x' \vert.
\end{equation*}
\end{proof}
The following part is devoted to showing that 
\begin{equation*}
	\|  E_{n+1} - E_n  \|_{a,t_0}  \leq \frac{1}{2} \|  E_{n} - E_{n-1}  \|_{a,t_0}.
\end{equation*}

\begin{lem}
\label{lemma X_n+1 - X_n}
Let $X_{n+1}(t)$ and $X_n(t)$ be two solutions of the characteristic curves (\ref{cc_VPME}) with vector fields $E_n$ and $E_{n-1}$ respectively and with the same asymptotic conditions.   That is
\begin{equation*}
	  \left \{
    \begin{array}{l}
    \dot{X}_{n+1} (t) = V_{n+1}(t),\\
    \dot{V}_{n+1}(t) = \bar{E}_n(t) + \widetilde{E}_n(t) = E_n(t),\\
    \lim_{t \to \infty} X_{n+1}(t) - V_{n+1}(t)t = x,\\
    \lim_{t \to \infty} V_{n+1}(t) = v,
    \end{array}
    \right.
    \qquad \mbox{and} \qquad
      \left \{
    \begin{array}{l}
    \dot{X}_n (t) = V_n(t),\\
    \dot{V}_n(t) = \bar{E}_{n-1}(t) + \widetilde{E}_{n-1}(t) = E_{n-1}(t),\\
    \lim_{t \to \infty} X_n(t) - V_n(t)t = x,\\
    \lim_{t \to \infty} V_n(t) = v.
    \end{array}
    \right.
\end{equation*}
Then 
\begin{equation}
\label{X_n+1 - X_n}
	 \|  X_{n+1} - X_n  \|_{a,t_0}   \leq \frac{1}{a^2- (24a_2 + 3)} \|  E_n - E_{n-1}  \|_{a,t_0}.
\end{equation}
\end{lem}

\begin{proof}
This proof follows step 6 in  \cite[Lemma 3.1]{EC_CM}. The integral solutions of $X_{n+1}(t)$ and $X_n(t)$ are given by 
\begin{equation*}
	X_{n+1}(t,x,v) = x+vt + \int_t^\infty (s-t) E_n (s, X_{n+1}(s,x,v)) \ ds
\end{equation*}
and 
\begin{equation*}
	X_{n}(t,x,v) = x+vt + \int_t^\infty (s-t) E_{n-1} (s, X_{n}(s,x,v)) \ ds.
\end{equation*}
First, we can show that $\left| X_{n+1}(t,x,v) - X_{n}(t,x,v) \right|$ is bounded. Indeed,
\begin{align}
\label{X_n bdd}
	\left| X_{n+1}(t,x,v) - X_{n}(t,x,v) \right| & \leq  \int_t^\infty (s-t) \vert E_n (s, X_{n+1}(s,x,v))-E_{n-1} (s, X_{n}(s,x,v)) \vert \ ds \nonumber \\
	& \leq  \int_t^\infty (s-t) e^{-as} ( \|  E_n  \|_{a,t_0}  + \|  E_{n-1}  \|_{a,t_0} )  \ ds  \nonumber \\
	& = \frac{e^{-at}}{a^2} ( \|  E_n  \|_{a,t_0}  + \|  E_{n-1}  \|_{a,t_0} ),
\end{align}
which is bounded because by assumption $E_n$ and $E_{n-1}$ are bounded by (\ref{bdd E_n}).
Next we have, 
\begin{align*}
	\left| X_{n+1}(t,x,v) - X_{n}(t,x,v) \right| 	& \leq  \int_t^\infty (s-t) \vert E_n (s, X_{n+1}(s,x,v))-E_{n-1} (s, X_{n+1}(s,x,v)) \vert \ ds \\
	& \quad +  \int_t^\infty (s-t) \vert E_{n-1} (s, X_{n+1}(s,x,v))-E_{n-1} (s, X_{n}(s,x,v)) \vert \ ds  \\
	& \leq  \int_t^\infty (s-t) e^{-as}  \|  E_n  - E_{n-1}   \|_{a,t_0}  \ ds \\
	& \quad +  \int_t^\infty (s-t) (24 a_2 + 3)  \vert X_{n+1}(s,x,v)- X_{n}(s,x,v) \vert \ ds  \\
	& \leq  \frac{e^{-at}}{a^2}  \|  E_n  - E_{n-1}   \|_{a,t_0}  + (24 a_2 + 3) \int_t^\infty (s-t)  \vert X_{n+1}(s,x,v)- X_{n}(s,x,v) \vert \ ds,
\end{align*}
where we used the Lipschitz property of the field (\ref{E_n Lipschitz}) and the definition \ref{def_norm} for the norm. We can bootstrap the last inequality as follow,
\begin{align*}
	& \left| X_{n+1}(t,x,v)   -  X_{n}(t,x,v) \right| 	 \leq  \frac{e^{-at}}{a^2}  \|  E_n  - E_{n-1}   \|_{a,t_0} + (24 a_2 + 3) \int_t^\infty (s-t)  \left(  \frac{e^{-as}}{a^2}  \|  E_n  - E_{n-1}   \|_{a,t_0} \right) \ ds  \\
	&   + (24 a_2 + 3)^2 \int_t^\infty (s-t) \int_s^\infty (s_1-s)   \vert X_{n+1}(s_1,x,v)- X_{n}(s_1,x,v) \vert \ ds_1   \ ds \\
	& \leq  \frac{e^{-at}}{a^2}  \|  E_n  - E_{n-1}   \|_{a,t_0}  \sum_{k=0}^K \frac{(24 a_2 + 3)^k}{a^{2k}} \\
	&  + (24 a_2 + 3)^{K+1} \int_t^\infty (s-t) \int_s^\infty (s_1-s) \  ...  \int_{s_{K-1}}^\infty (s_{K}-s_{K-1})   \vert X_{n+1}(s_{K},x,v)- X_{n}(s_{K},x,v) \vert \ ds_{K} ...  \ ds_1  \ ds \\ 
	& \leq \frac{e^{-at}}{a^2}  \|  E_n  - E_{n-1}   \|_{a,t_0}  \sum_{k=0}^K \frac{(24 a_2 + 3)^k}{a^{2k}}  \quad + \frac{(24 a_2 + 3)^{K+1}}{a^{2(K+1)}} \frac{e^{-at}}{a^2} ( \|  E_n  \|_{a,t_0}  + \|  E_{n-1}  \|_{a,t_0} ),
\end{align*}
where we applied (\ref{X_n bdd}) to bound $\vert X_{n+1}(s_{K},x,v)- X_{n}(s_{K},x,v) \vert $ in the last integral.  Recalling that $a^2 \geq (112 a_2 + 3) (4e^{12}+1) > 24a_2 + 3$ we can take the limit $K$ to infinity. Therefore, using the geometric series and the fact that the last term vanishes in the limit, we get
\begin{align*}
	\left| X_{n+1}(t,x,v) - X_{n}(t,x,v) \right| 	& \leq \frac{e^{-at}}{a^2}  \|  E_n  - E_{n-1}   \|_{a,t_0}  \sum_{k=1}^\infty \frac{(24 a_2 + 3)^k}{a^{2k}} \\
	& =  \frac{e^{-at}}{a^2}  \|  E_n  - E_{n-1}   \|_{a,t_0} \frac{1}{1 - \frac{(24 a_2 + 3)}{a^2}} \\
	& =  \frac{e^{-at}}{a^2 - (24 a_2 + 3)}  \|  E_n  - E_{n-1}   \|_{a,t_0}.
\end{align*}
Finally, by definition \ref{def_norm}, it follows 
\begin{equation*}
	 \|  X_{n+1} - X_n  \|_{a,t_0}   \leq \frac{1}{a^2- (24a_2 + 3)} \|  E_n - E_{n-1}  \|_{a,t_0}.
\end{equation*}
\end{proof}
Using this lemma we can show the following result,
\begin{lem}
\label{lem_E_n+1 - E_n}
Let $E_i = \bar{E}_i + \widetilde{E}_i$ for $i=n$ and $n+1$. Then 
\begin{equation}
\label{E_n+1 - E_n (bar)}
	 \| \bar{E}_{n+1} - \bar{E}_n  \|_{a,t_0} \leq \frac{44 a_2}{a^2-(24a_2 + 3)} \|  E_n - E_{n-1}  \|_{a,t_0}.
\end{equation}
\end{lem}
\begin{proof}
This proof follows step 7 in  \cite[Lemma 3.1]{EC_CM}. We use the representation of $\bar{E}_{i}$ as a convolution with the density $\rho_i$ and $W'$.
\begin{align*}
	\bar{E}_{i} (t,x) & = -  \int_\bt W'(x-y) \rho_{i}(t,y) \ dy  =  - \int_{\bt \times \br} W'(x-y) f^* (\phi_t^{-1} (y,v))  \ dy \ dv \\
	& = - \int_{\bt \times \br} W'(x-X_{i}(t,y,v)) f^* (y,v)  \ dy \ dv,
\end{align*}
where we used that $\phi_t (y,v) = (X_{i}(t,y,v), V_{i}(t,y,v))$ is measure preserving for the change of variable. Next,  we define 
\begin{equation*}
	\varepsilon : =\|  X_{n+1} (t) - X_n (t) \|_{L^{\infty} (\bt \times \br)},
\end{equation*}
and the two following sets
\begin{align*}
	A_1 = \{ (y,v) \in \bt \times \br : \left| x - X_{n+1}(t,y,v)  \right|  \geq 2 \varepsilon \} \quad \mbox{ and } \quad A_2 = \{ (y,v) \in \bt \times \br : \left| x - X_{n+1}(t,y,v)  \right|  <  2 \varepsilon \}.
\end{align*}
Therefore, we have
\begin{align*}
	\left| \bar{E}_{n+1} (t,x) - \bar{E}_{n} (t,x) \right|  & \leq \int_{\bt \times \br} \left| W'(x-X_{n+1}(t,y,v)) -W'(x-X_{n}(t,y,v))  \right| f^* (y,v)  \ dy \ dv \\
	&  \leq \int_{A_1} \left| W'(x-X_{n+1}(t,y,v)) -W'(x-X_{n}(t,y,v))  \right| f^* (y,v)  \ dy \ dv \\
	& \quad +   \int_{A_2} \left| W'(x-X_{n+1}(t,y,v)) -W'(x-X_{n}(t,y,v))  \right| f^* (y,v)  \ dy \ dv \\
	& =:I_{A_1} + I_{A_2}.
\end{align*}
We consider the line segment connecting $x-X_{n+1}(t,y,v)$ and $x-X_{n}(t,y,v)$.  We observe that on $A_1$, we have for all $s \in [0,1]$
\begin{align*}
	\left| (1-s) (x-X_{n+1}(t,y,v))  + s (x-X_{n}(t,y,v))\right| & = \left|  x-X_{n+1}(t,y,v)  - s (X_{n}(t,y,v) -X_{n+1}(t,y,v))\right| \\
	& \geq \left|  x-X_{n+1}(t,y,v) \right|  - s \left| X_{n}(t,y,v) -X_{n+1}(t,y,v) \right| \\
	& \geq \left|  x-X_{n+1}(t,y,v) \right|  - s \varepsilon  \\
	& \geq  \varepsilon.
\end{align*}
Therefore, on $A_1$, we can use the Lipschitz condition (\ref{lipschitz_W'}) of $W'$.
Hence,
\begin{align*}
	I_{A_1} & \leq \int_{A_1} \left| X_{n+1}(t,y,v) -X_{n}(t,y,v)  \right| f^* (y,v)  \ dy \ dv \\
	& \leq \|  X_{n+1} (t) - X_n (t)  \|_{L^{\infty} (\bt \times \br)}  \int_{\bt \times \br}  f^* (y,v)  \ dy \ dv \\
	& \leq  4 a_2 \|  X_{n+1} (t) - X_n (t)  \|_{L^{\infty} (\bt \times \br)}.
\end{align*}
where we used (\ref{bdd f*}) to bound the $L^1$ norm of $f^*$. For $I_{A_2}$, we use the fact that $W'$ is bounded, i.e.  inequality (\ref{bound_W'}). Therefore, we obtain
\begin{equation*}
	I_{A_2} \leq 2 \int_{A_2}  f^* (y,v)  \ dy \ dv = 2 \int_{\bt \times \br} f^* (y,v) \mathbbm{1}_{\{ \left| x - X_{n+1}(t,y,v)  \right|  <  2 \varepsilon \}}   \ dy \ dv.
\end{equation*}
Then, using that the flow is measure preserving we get 
\begin{align*}
	I_{A_2} & \leq  2 \int_{\bt \times \br} f^* (\phi_t^{-1}(y,v)) \mathbbm{1}_{\{ \left| x - y  \right|  <  2 \varepsilon \}}   \ dy \ dv \\
	& \leq 2 \| \rho_{n+1} \|_{L^{\infty} (\bt)} \int_{\bt}  \mathbbm{1}_{\{ \left| x - y  \right|  <  2 \varepsilon \}}   \ dy \\
	& \leq 40 a_2 \|  X_{n+1} (t) - X_n (t) \|_{L^{\infty} (\bt \times \br)},
\end{align*}
where we used the definition of $\varepsilon$ and (\ref{bound_rho_n+1}) in the last inequality.  Summing $I_{A_1}$ and $I_{A_2}$ and by the definition \ref{def_norm},  we get
\begin{align*}
	\| \bar{E}_{n+1} - \bar{E}_{n} \|_{a,t_0}  & \leq 44 a_2 \|  X_{n+1} - X_n  \|_{a,t_0} \leq \frac{44a_2}{a^2-(24a_2 + 3)} \|  E_n - E_{n-1}  \|_{a,t_0},
\end{align*}
where we used (\ref{X_n+1 - X_n}) for the last inequality.
\end{proof}

\begin{lem}
Let $E_i = \bar{E}_i + \widetilde{E}_i$ and $U_i = \bar{U}_i + \widetilde{U}_i$ for $i=n$ and $n+1$. Then 
	\begin{equation}
\label{E_n+1 - E_n (hat)}
	 \|  \widetilde{E}_{n+1} - \widetilde{E}_n  \|_{a,t_0}  \leq 4e^{12} \|  \bar{U}_{n+1} - \bar{U}_n \|_{a,t_0}.
\end{equation}
\end{lem}
\begin{proof}
This lemma is a direct consequence of inequality (\ref{U'_1 - U'_2 hat}). By definition we have,
\begin{align*}
	 \|  \widetilde{E}_{n+1} - \widetilde{E}_n  \|_{a,t_0} &= \sup_{t \geq t_0} e^{at} \|  \widetilde{E}_{n+1} (t)- \widetilde{E}_n (t)  \|_{L^{\infty}(\bt)} = \sup_{t \geq t_0} e^{at} \|  \pa_x \widetilde{U}_{n+1} (t)- \pa_x \widetilde{U}_n (t)  \|_{L^{\infty}(\bt)} \\
	 & \leq \sup_{t \geq t_0} e^{at}  4 e^{12} \| \bar{U}_{n+1}(t)  - \bar{U}_n(t)  \|_{L^{\infty} (\bt)} = 4 e^{12} \|  \bar{U}_{n+1} - \bar{U}_n \|_{a,t_0}.
\end{align*}
\end{proof}

We now estimate $\|  \bar{U}_{n+1} - \bar{U}_n \|_{a,t_0}$. 
\begin{lem}
Let $U_i = \bar{U}_i + \widetilde{U}_i$ for $i=n$ and $n+1$. Recall that $\bar{E}_i = - \pa_x \bar{U}_i$ and 
\begin{equation*}
	 \bar{U}_i (t,x) = \int_\bt W(x-y) \rho_i(t,y) \ dy,
\end{equation*} 
with 
\begin{equation*}
	W(x) = \frac{x^2-\vert x \vert}{2}.
\end{equation*}
Then 
\begin{equation}
\label{U_n+1 - U_n (bar)}
	\|  \bar{U}_{n+1} - \bar{U}_n  \|_{a,t_0} \leq \frac{44a_2}{a^2-(24a_2 + 3)} \|  E_n - E_{n-1}  \|_{a,t_0}.
\end{equation}
\end{lem}
\begin{proof}
The proof is similar to the one of Lemma \ref{lem_E_n+1 - E_n}. We use 
\begin{equation*}
	 \bar{U}_i (t,x) = \int_\bt W(x-y) \rho_i(t,y) \ dy = \int_\bt W(x-X_{i}(t,y,v)) f^* (y,v)  \ dy \ dv,
\end{equation*}
and the facts that $\| W \|_{L^{\infty}(\bt)} \leq 1$ and $W$ is 1-Lipschitz.  Therefore, we can apply the same estimates as in Lemma \ref{lem_E_n+1 - E_n} to have the desired result.
\end{proof}

With these lemmas above, we can now state the main result of this section.
\begin{prop} The map $E_n \mapsto E_{n+1}$ is a contraction with respect to the norm $\| \cdot  \|_{a,t_0} $. That is,
\begin{equation}
\label{E_n+1 - E_n}
	\|  E_{n+1} - E_n  \|_{a,t_0} \leq  \frac{1}{2} \|  E_n - E_{n-1}  \|_{a,t_0}.
\end{equation}
\end{prop}
\begin{proof}
Using (\ref{E_n+1 - E_n (bar)}), (\ref{E_n+1 - E_n (hat)}) and (\ref{U_n+1 - U_n (bar)}), we have 
\begin{align*}
	\|  E_{n+1} - E_n  \|_{a,t_0} & \leq \|  \widetilde{E}_{n+1} - \widetilde{E}_n  \|_{a,t_0} +  \|  \bar{E}_{n+1} - \bar{E}_n  \|_{a,t_0}  \\
	& \leq 4 e^{12} \|  \bar{U}_{n+1} - \bar{U}_n \|_{a,t_0} +  \frac{44a_2}{a^2-(24a_2 + 3)} \|  E_n - E_{n-1}  \|_{a,t_0} \\
	& \leq 4 e^{12}   \frac{44a_2}{a^2-(24a_2 + 3)} \|  E_n - E_{n-1}  \|_{a,t_0} +  \frac{44a_2}{a^2-(24a_2 + 3)} \|  E_n - E_{n-1}  \|_{a,t_0} \\
	& \leq \frac{1}{2}  \|  E_n - E_{n-1}  \|_{a,t_0}.
\end{align*}
Indeed, by assumption of Theorem \ref{thm1}, we have $a^2 \geq (112 a_2 + 3) (4 e^{12}+1)$, therefore
\begin{align*}
	(4 e^{12} + 1)  \frac{44a_2}{a^2-(24a_2 + 3)} & \leq (4 e^{12} + 1) \frac{44a_2}{a^2-(24a_2 + 3)(4 e^{12} + 1)} \\
	& \leq (4 e^{12} + 1) \frac{44a_2}{(112 a_2 + 3) (4 e^{12}+1)-(24a_2 + 3)(4 e^{12} + 1)} \\
	& =  (4 e^{12} + 1) \frac{44a_2}{88 a_2 (4 e^{12}+1)}  = \frac{1}{2}.
\end{align*}
\end{proof}

With  inequalities (\ref{bdd E_1}) and (\ref{E_n+1 - E_n}), we can now prove (\ref{bdd E_n - E_n-1}),
\begin{equation*}
	\|  E_{n+1} - E_n  \|_{a,t_0} \leq  \frac{1}{2^n}  \|  E_{1} - E_0  \|_{a,t_0} =  \frac{1}{2^n}  \|  E_{1}  \|_{a,t_0} \leq  \frac{1}{2^n} 8 a_1.
\end{equation*}
To show (\ref{bdd E_n}) we use 
\begin{align*}
	\|  E_{n+1}  \|_{a,t_0} & \leq \|  E_{n+1} - E_n  \|_{a,t_0} + \|  E_n  \|_{a,t_0}  \leq  \frac{1}{2^n} 8 a_1  + \|  E_n - E_{n-1}  \|_{a,t_0} +  \|  E_{n-1}  \|_{a,t_0} \\
	& \leq  \frac{1}{2^n} 8 a_1 +  \frac{1}{2^{n-1}} 8 a_1  +  \|  E_{n-1} - E_{n-2}  \|_{a,t_0} + \|  E_{n-2}  \|_{a,t_0}  \leq \sum_{i=1}^{n} \frac{1}{2^i} 8 a_1 + \|  E_{1}  \|_{a,t_0}.
\end{align*}
Therefore, with (\ref{bdd E_1}), we get the desired bound
\begin{align*}
	\|  E_{n+1}  \|_{a,t_0}  \leq \sum_{i=1}^{n} \frac{1}{2^i} 8 a_1  + 8 a_1  \leq 16  a_1.
\end{align*}

\subsection{Convergence of the iterative scheme}

By (\ref{bdd E_n - E_n-1}), the sequence $E_n$ converges to a bounded and  Lipschitz function $E$,  i.e.
\begin{align*}
	\|  E  \|_{a,t_0}  \leq 16 a_1  \qquad \mbox{and} \qquad \vert E (t,x) - E (t, x') \vert  \leq (24 a_2+3) \vert x-x' \vert.
\end{align*}
Therefore, we have existence and uniqueness of a solution $f$ given by 
\begin{equation*}
	f(t,x,v) = f^*(\phi_t^{-1} (x,v)).
\end{equation*}
Moreover, with the same reasoning as in (\ref{computation asymptotic}),  the asymptotic condition (\ref{asymptotic_cond}) is satisfied.  And the solution $f$ becomes homogeneous. Indeed, for every test functions $\varphi \in C_c^0 (\Omega)$ we have 
	\begin{align*}
		 \int_\Omega \varphi (x,v) f(t,x,v) \ dx  \ dv & = \int_\Omega \varphi (x,v) (f(t,x,v) - f^*(x-vt, v)) \ dx  \ dv  + \int_\Omega \varphi (x,v) f^*(x-vt, v) \ dx  \ dv.
	\end{align*}
The first term vanishes when $t$ goes to infinity thanks to (\ref{asymptotic_cond}). The second term converges weakly to 
\begin{equation*}
	h(v) = \int_\bt f^*(x,v) \ dx,
\end{equation*}
because $f^*(x-vt, v)$ is a solution of the free transport (see e.g. \cite[Theorem 2.1]{EC_CM}).

\section{Instability result for a class of stationary solutions}
\label{section_unstable}
In this section we show that a certain class of stationary solutions are unstable in a weak topology.  

\begin{prop}
Let $\mu(v)$ be a stationary solution of the system (\ref{VPME}) such that for some constants $a$ and $a_2$ with $a^2 \geq (112 a_2 + 3) (4 e^{12}+1)$, we have 
\begin{itemize}
\item[i)] $\mu$ is analytic, i.e.
\begin{equation}
\label{Fourier transform mu}
	\vert \widehat{\mu}(\eta)  \vert \leq \frac{e^{-12}}{4}  e^{-a\vert \eta \vert}.
\end{equation}
\item[ii)] $\mu$ decays fast enough at infinity, i.e.
\begin{equation}
\label{bdd mu}
	\vert \mu (v) \vert \leq \frac{a_2}{2(1+v^4)}.
\end{equation}
\item[iii)] $\mu(v) \geq 0$.
\end{itemize}
Then $\mu(v)$ is an unstable stationary solution in the weak topology.
\end{prop}

\begin{proof}
Let us define the asymptotic profile $f^*(x,v) = \mu(v) (1 + \cos( 2 \pi x))$.  Clearly, $f^* \geq 0$ and
\begin{align*}
	 \widehat{f^*}(k,\eta) & =  \int_\bt \int_\br f^*(x,v) e^{-2 \pi ikx} e^{-i \eta v} \ dx \ dv = \int_\bt (1 + \cos(2 \pi x)) e^{-2 \pi ikx} \ dx  \int_\br \mu(v) e^{-i \eta v}  \ dv \\
	 & = \big(\delta_{k=0} + \frac{1}{2}( \delta_{k=-1} + \delta_{k=1}) \big)\widehat{\mu}(\eta).
\end{align*}
Therefore,  by (\ref{Fourier transform mu}),  $ \widehat{f^*}(k,\eta)$ satisfies  assumption (\ref{Fourier transform f*}). Moreover, by (\ref{bdd mu}), $f^*$ satisfies
\begin{align*}
	\vert f^* (x,v) \vert \leq (1 + 1) \vert \mu (v) \vert \leq  \frac{a_2}{1+v^4}.
\end{align*}
Hence, $f^*$ fulfils all the assumptions of Theorem \ref{thm1}.  It follows, by Theorem \ref{thm1}, the existence of a solution $f=f(t,x,v)$ of the system (\ref{VPME}) such that the asymptotic condition (\ref{asymptotic_cond}) is satisfied. Moreover, $f$ converges weakly to $\mu(v)$. Indeed,  by a direct computation, we obtain
\begin{equation*}
	\int_\bt f^* (x,v) \ dx = \int_\bt \mu (v) (1 +  \cos(2 \pi x)) \ dx  =  \mu(v),
\end{equation*}
where we used $ \int_\bt \cos(2 \pi x) \ dx = 0$.

We can therefore construct an initial condition for $f$ as close as we want to $\mu(-v)$ in the weak topology.  Then, by the time reversibility of the system (\ref{VPME}), we can invert time and velocity. Hence, the solution $f=f(-t,x,-v)$  will not stay close to $\mu(-v)$ after a certain time in the weak topology.
\end{proof}

\textbf{Acknowledgements:} The author would like to warmly thank his advisor, Mikaela Iacobelli, for offering him to work on this subject and for her constant help and guidance. The author is also very grateful to the anonymous referees for their constructive reviews.

\bibliographystyle{plain}
\bibliography{Scattering_Landau_damping_VPME}

\begin{thebibliography}{10}

\bibitem{Arsenev}
A.~A. Arsenev.
\newblock Existence in the large of a weak solution of {V}lasov's system of
  equations.
\newblock {\em \v{Z}. Vy\v{c}isl. Mat i Mat. Fiz.}, 15:136--147, 276, 1975.

\bibitem{Bardos_Degond_1}
C.~Bardos and P.~Degond.
\newblock Existence globale des solutions des \'{e}quations de
  {V}lasov-{P}oisson.
\newblock In {\em Nonlinear partial differential equations and their
  applications. {C}oll\`ege de {F}rance seminar, {V}ol. {VII} ({P}aris,
  1983--1984)}, volume 122 of {\em Res. Notes in Math.}, pages 1--3, 35--58.
  Pitman, Boston, MA, 1985.

\bibitem{Bardos_Degond_2}
C.~Bardos and P.~Degond.
\newblock Global existence for the {V}lasov-{P}oisson equation in {$3$} space
  variables with small initial data.
\newblock {\em Ann. Inst. H. Poincar\'{e} Anal. Non Lin\'{e}aire},
  2(2):101--118, 1985.

\bibitem{Bardos_Degond_Golse}
C.~Bardos, P.~Degond, and F.~Golse.
\newblock A priori estimates and existence results for the {V}lasov and
  {B}oltzmann equations.
\newblock In {\em Nonlinear systems of partial differential equations in
  applied mathematics, {P}art 2 ({S}anta {F}e, {N}.{M}., 1984)}, volume~23 of
  {\em Lectures in Appl. Math.}, pages 189--207. Amer. Math. Soc., Providence,
  RI, 1986.

\bibitem{Bardos}
C.~Bardos, F.~Golse, T.~Nguyen, and R.~Sentis.
\newblock The {M}axwell-{B}oltzmann approximation for ion kinetic modeling.
\newblock {\em Phys. D}, 376/377:94--107, 2018.

\bibitem{Batt_Rein}
J.~Batt and G.~Rein.
\newblock Global classical solutions of the periodic {V}lasov-{P}oisson system
  in three dimensions.
\newblock {\em C. R. Acad. Sci. Paris S\'{e}r. I Math.}, 313(6):411--416, 1991.

\bibitem{BMM_13}
J.~Bedrossian, N.~Masmoudi, and C.~Mouhot.
\newblock Landau damping: paraproducts and {G}evrey regularity.
\newblock {\em Ann. PDE}, 2(1):Art. 4, 71, 2016.

\bibitem{BCR}
D.~Benedetto, E.~Caglioti, and S.~Rossi.
\newblock Comparison between the {C}auchy problem and the scattering problem
  for the {L}andau damping in the {V}lasov-{HMF} equation, 03 2021.
\newblock arXiv:2103.15932.

\bibitem{Bouchut}
F.~Bouchut.
\newblock Global weak solution of the {V}lasov-{P}oisson system for small
  electrons mass.
\newblock {\em Comm. Partial Differential Equations}, 16(8-9):1337--1365, 1991.

\bibitem{plasma_2003}
T.~J.~M. Boyd and J.~J. Sanderson.
\newblock {\em The Physics of Plasmas}.
\newblock Cambridge University Press, 2003.

\bibitem{EC_CM}
E.~Caglioti and C.~Maffei.
\newblock Time asymptotics for solutions of {V}lasov-{P}oisson equation in a
  circle.
\newblock {\em J. Statist. Phys.}, 92(1-2):301--323, 1998.

\bibitem{Cesbron_Iacobelli}
L.~Cesbron and M.~Iacobelli.
\newblock Global well-posedness of {V}lasov-{P}oisson-type systems in bounded
  domains, 08 2021.
\newblock arXiv:2108.11209v1.

\bibitem{Chen}
Z.~Chen and J.~Chen.
\newblock Moments propagation for weak solutions of the {V}lasov-{P}oisson
  system in the three-dimensional torus.
\newblock {\em J. Math. Anal. Appl.}, 472(1):728--737, 2019.

\bibitem{GI}
A.~Gagnebin and M.~Iacobelli.
\newblock Landau damping on the torus for the {V}lasov-{P}oisson system with
  massless electrons, 09 2022.
\newblock arXiv:2209.04676.

\bibitem{GNR}
E.~Grenier, T.~Nguyen, and I.~Rodnianski.
\newblock Landau damping for analytic and {G}evrey data.
\newblock 04 2020.
\newblock arXiv:2004.05979.

\bibitem{Griffin_Iacobelli_R}
M.~Griffin-Pickering and M.~Iacobelli.
\newblock Global strong solutions in {$\mathbb{R}^3$} for ionic
  {V}lasov-{P}oisson systems.
\newblock {\em Kinet. Relat. Models}, 14(4):571--597, 2021.

\bibitem{Griffin_Iacobelli_torus}
M.~Griffin-Pickering and M.~Iacobelli.
\newblock Global well-posedness for the {V}lasov-{P}oisson system with massless
  electrons in the 3-dimensional torus.
\newblock {\em Comm. Partial Differential Equations}, 46(10):1892--1939, 2021.

\bibitem{Griffin_Iacobelli_summary}
M.~Griffin-Pickering and M.~Iacobelli.
\newblock Recent developments on the well-posedness theory for {V}lasov-type
  equations.
\newblock In {\em From particle systems to partial differential equations},
  volume 352 of {\em Springer Proc. Math. Stat.}, pages 301--319. Springer,
  Cham, 2021.

\bibitem{Han-Kwan_Iacobelli}
D.~Han-Kwan and M.~Iacobelli.
\newblock The quasineutral limit of the {V}lasov-{P}oisson equation in
  {W}asserstein metric.
\newblock {\em Commun. Math. Sci.}, 15(2):481--509, 2017.

\bibitem{Horst_Hunze}
E.~Horst and R.~Hunze.
\newblock Weak solutions of the initial value problem for the unmodified
  nonlinear {V}lasov equation.
\newblock {\em Math. Methods Appl. Sci.}, 6(2):262--279, 1984.

\bibitem{huang_2d}
L.~Huang, Q.-H. Nguyen, and Y.~Xu.
\newblock Nonlinear {Landau} damping for the 2d {Vlasov}-{Poisson} system with
  massless electrons around {Penrose}-stable equilibria.
\newblock 06 2022.
\newblock arXiv:2206.11744.

\bibitem{huang_sharp_2022}
L.~Huang, Q.-H. Nguyen, and Y.~Xu.
\newblock Sharp estimates for screened {Vlasov}-{Poisson} system around
  {Penrose}-stable equilibria in $\mathbb{R}^d$, $d\geq3$, 05 2022.
\newblock arXiv:2205.10261.

\bibitem{Hw_Ve}
H.~J. Hwang and J.~L.~L. Vel\'{a}zquez.
\newblock On the existence of exponentially decreasing solutions of the
  nonlinear {L}andau damping problem.
\newblock {\em Indiana Univ. Math. J.}, 58(6):2623--2660, 2009.

\bibitem{Iordanskii}
S.~V. Iordanski\u{\i}.
\newblock The {C}auchy problem for the kinetic equation of plasma.
\newblock {\em Trudy Mat. Inst. Steklov.}, 60:181--194, 1961.

\bibitem{Landau}
L.~Landau.
\newblock On the vibrations of the electronic plasma.
\newblock {\em (Russian) Akad. Nauk SSSR. Zhurnal Eksper. Teoret. Fiz},
  16:574--586, 1946.

\bibitem{Lions_Perthame}
P.-L. Lions and B.~Perthame.
\newblock Propagation of moments and regularity for the {$3$}-dimensional
  {V}lasov-{P}oisson system.
\newblock {\em Invent. Math.}, 105(2):415--430, 1991.

\bibitem{Loeper}
G.~Loeper.
\newblock Uniqueness of the solution to the {V}lasov-{P}oisson system with
  bounded density.
\newblock {\em J. Math. Pures Appl. (9)}, 86(1):68--79, 2006.

\bibitem{Miot}
Evelyne Miot.
\newblock A uniqueness criterion for unbounded solutions to the
  {V}lasov-{P}oisson system.
\newblock {\em Comm. Math. Phys.}, 346(2):469--482, 2016.

\bibitem{CM_CV}
C.~Mouhot and C.~Villani.
\newblock On {L}andau damping.
\newblock {\em Acta Math.}, 207(1):29--201, 2011.

\bibitem{Pallard}
C.~Pallard.
\newblock Moment propagation for weak solutions to the {V}lasov-{P}oisson
  system.
\newblock {\em Comm. Partial Differential Equations}, 37(7):1273--1285, 2012.

\bibitem{Penrose}
O.~Penrose.
\newblock Electrostatic instabilities of a uniform non-maxwellian plasma.
\newblock {\em Physics of Fluids (U.S.)}, Vol: 3, 3 1960.

\bibitem{Pfaffelmoser}
K.~Pfaffelmoser.
\newblock Global classical solutions of the {V}lasov-{P}oisson system in three
  dimensions for general initial data.
\newblock {\em J. Differential Equations}, 95(2):281--303, 1992.

\bibitem{Ryutov_1999}
D.~D. Ryutov.
\newblock Landau damping: half a century with the great discovery.
\newblock {\em Plasma Physics and Controlled Fusion}, 41(3A):A1--A12, jan 1999.

\bibitem{Schaffer}
J.~Schaeffer.
\newblock Global existence of smooth solutions to the {V}lasov-{P}oisson system
  in three dimensions.
\newblock {\em Comm. Partial Differential Equations}, 16(8-9):1313--1335, 1991.

\bibitem{Ukai_Okabe}
S.~Ukai and T.~Okabe.
\newblock On classical solutions in the large in time of two-dimensional
  {V}lasov's equation.
\newblock {\em Osaka Math. J.}, 15(2):245--261, 1978.

\bibitem{Villani_notes}
C.~Villani.
\newblock Landau damping, {N}otes de cours.
\newblock CEMRACS, 2010.
\newblock
  \url{http://www.cedricvillani.org/sites/dev/files/old_images/2012/08/B13.Landau.pdf}.

\end{thebibliography}

\end{document}